\documentclass[draft,leqno,11pt]{amsart}
\usepackage{mathrsfs}
\usepackage{amssymb}
\usepackage{amsmath}
\usepackage{amsfonts}
\usepackage{enumerate}
\usepackage{graphics}
\usepackage{verbatim}
\usepackage{amsthm}
\usepackage{latexsym,bm}

\input xypic
\xyoption {all}

\newtheorem{theorem}{Theorem}[section]
\newtheorem{corollary}[theorem]{Corollary}
\newtheorem{lemma}[theorem]{Lemma}
\newtheorem{proposition}[theorem]{Proposition}
\theoremstyle{definition}
\newtheorem{definition}[theorem]{Definition}
\newtheorem{remark}[theorem]{Remark}
\newtheorem{remarks}[theorem]{Remarks}
\newtheorem{example}[theorem]{Example}
\newtheorem{problem}{Problem}

\newcommand{\h}{\mathsf{h}}
\newcommand{\at}{\mathrm{at}}
\newcommand{\real}{{\mathbb R}}
\newcommand{\nat}{{\mathbb N}}
\newcommand{\T}{{\tau}}
\newcommand{\E}{{\mathcal E}}

\newcommand{\M}{{\mathcal M}}
\newcommand{\N}{{\mathcal N}}
\newcommand{\R}{{\mathcal R}}

\newcommand{\BMO}{{\mathcal {BMO}}}
\newcommand{\g}{\gamma}

\newcommand{\tr}{\mbox{\rm tr}}
\newcommand{\wt}{\widetilde}

\numberwithin{equation}{section}

\begin{document}

\title{Noncommutative  fractional integrals}

\author[N. Randrianantoanina]{Narcisse Randrianantoanina}

\address{Department of Mathematics, Miami University, Oxford, OH  45056, USA}

\email{randrin@miamioh.edu}

\author[L. Wu]{Lian Wu}
\address{Institute of Probability and Statistics, Central South University, Changsha 410075, China}
\curraddr{Department of Mathematics, Miami University, Oxford, OH 45056, USA}
\email{wulian567@126.com, wul5@miamioh.edu}



\thanks{ Wu was partially supported by the China Scholarship Council.}

\begin{abstract}
Let $\M$ be a hyperfinite finite von Nemann algebra and $(\M_k)_{k\geq 1}$ be an increasing filtration of finite dimensional von Neumann subalgebras of $\M$. We investigate  abstract fractional integrals associated to the filtration $(\M_k)_{k\geq 1}$. For  a finite noncommutative   martingale $x=(x_k)_{1\leq k\leq n} \subseteq  L_1(\M)$ adapted to $(\M_k)_{k\geq 1}$ and $0<\alpha<1$,   the fractional integral of  $x$  of order $\alpha$ is defined  by setting:
$$I^\alpha x = \sum_{k=1}^n \zeta_k^{\alpha } dx_k$$
for an  appropriate sequence  of scalars $(\zeta_k)_{k\geq 1}$.  For the case of  noncommutative dyadic  martingale in $L_1(\R)$ where 
 $\R$  is  the type ${\rm II}_1$ hyperfinite factor equipped with its natural  increasing filtration, $\zeta_k=2^{-k}$ for $k\geq 1$.

We prove that $I^\alpha$ is of weak-type $(1, 1/(1-\alpha))$. More precisely,  there is a constant ${\mathrm c}$ depending only on $\alpha$ such that if $x=(x_k)_{k\geq 1}$ is a finite noncommutative  martingale in $L_1(\M)$ then
\[
\|I^\alpha x\|_{L_{1/(1-\alpha),\infty}(\mathcal{\M})}\leq {\mathrm c}\|x\|_{L_1(\M)}.
\]
We   also obtain  that $I^\alpha$ is bounded  from $L_{p}(\M)$ into $L_{q}(\M)$  where $1<p<q<\infty$ and $\alpha=1/p-1/q$, thus providing a noncommutative analogue of a  classical result. Furthermore, we investigate  the corresponding result for  noncommutative martingale Hardy spaces. Namely,  there is  a constant ${\mathrm c}$ depending only on $\alpha$ such that if $x=(x_k)_{k\geq 1}$ is a finite noncommutative  martingale in  the martingale Hardy space $\mathcal{H}_1(\M)$ then
$\|I^\alpha x\|_{\mathcal{H}_{1/(1-\alpha)}(\M)}\leq  {\mathrm c} \|x\|_{\mathcal{H}_1(\M)}$.
\end{abstract}

\maketitle

\setcounter{section}{-1}


\section{Introduction}\label{Intro}


For $n\geq 1$, let $\mathcal{F}_n$ be the $\sigma$-algebra generated by dyadic intervals of length $2^{-n}$ in  the unit interval $[0,1]$, $\mathcal{F}$ be the $\sigma$-algebra generated by $\cup_{n\geq 1} \mathcal{F}_n$, and $\mathbb{P}$ denote the Lebesgue measure on $[0,1]$.
 A martingale $\{f_n\}_{n\geq 1}$  on the probability space $([0,1],\mathcal{F}, \mathbb{P})$ adapted to the increasing  filtration $\{\mathcal{F}_n\}_{n\geq 1}$ is called a dyadic martingale. The theory of dyadic martingales has played an important  role in the development of classical analysis such as harmonic analysis and Banach space theory. For instance, the connection between the study of Haar basis in rearrangement invariant spaces on $[0,1]$ and  dyadic martingales is quite obvious. The monograph by M\"ulcer \cite{Muller} contains a very detailed account  of dyadic martingale Hardy spaces and their applications in modern analysis.   Dyadic martingales  also appear naturally on various Littlewood-Paley type theory. We refer to the books \cite{EG,LT} for these  historical facts.  

Our primary interest in this article is closely related to  the so-called \emph{fractional integrals} for dyadic martingales. These are  special classes  of martingale transforms.  Let us review the basic classical setup. Given a dyadic martingale $f=\{f_n\}_{n\geq 1}$  and $0<\alpha <1$, the  dyadic fractional integral  (of order $\alpha$) of $f$ is the sequence  $I^\alpha f=\{ (I^\alpha f)_n\}_{n\geq 1}$   defined by setting:
\begin{equation}\label{fractional}
 (I^\alpha f)_n=\sum_{k=1}^n 2^{-k\alpha } df_k, \quad n\geq 1,
 \end{equation}
 where $\{df_k\}_{k\geq 1}$ is the martingale difference sequence of $f$. Dyadic fractional integrals  are closely related to  some particular types of Walsh-Fourier series. They also appear in various forms in function theory which goes back to Hardy and Littewood.
In \cite{Chao-Ombe}, Chao and Ombe   provided boundedness of fractional  integrals between various $L_p$-spaces depending on the size of $\alpha$. Their results can be summarized as follows:

\begin{theorem}[\cite{Chao-Ombe}]\label{classical} 
{\rm{(1)}} For  $1<p<q<\infty$ and $\alpha=1/p -1/q$, there exists a constant $C_{p,q}$ depending only  on $p$ and $q$ such that
\[
\big\| \sup_{n\geq 1} | (I^\alpha f)_n| \big\|_q \leq C_{p,q} \big\| f\big\|_p, \quad f \in L_p[0,1].
\]

{\rm{(2)}} For $0<\alpha<1$, there exists a constant $C_\alpha$ depending only on $\alpha$ such that for every $f\in L_1[0,1]$, 
\[
\mathbb{P}\big[ \sup_{n \geq 1} | (I^\alpha f)_n| \geq  \lambda \big] \leq C_\alpha\left(\frac{\|f\|_1}{\lambda}\right)^{1/(1-\alpha)}, \quad \forall \lambda>0.
\]
\end{theorem}
Recall that  martingale transforms are of strong-type $(p,p)$ for $1<p<\infty$  and of weak-type $(1,1)$. The  emphasis  here  is that the special nature of the coefficients in the fractional integrals  provides these $L_p$-$L_q$ type boundedness as opposed to  just the familiar  $L_p$-boundedness of martingale transforms.

Our primary objective  in this article is  to investigate possible  generalizations of  fractional integrals  in the general framework  of noncommutative martingales.
 This of course is part of  the general development of noncommutative  martingale theory for which we refer the reader  to \cite{PX,Ju,JX,Ran15} for recent history and results.   
We  will work with   general hyperfinite  finite von Neumann algebra $\M$  with increasing  filtration of finite dimensional subalgebras $(\M_n)_{n\geq 1}$.  We consider a unified approach to  fractional integrals for noncommutative martingales  adapted to  $(\M_n)_{n\geq 1}$.  These  abstract fractional integrals are of course  closely connected to the size of the filtration $(\M_n)_{n\geq 1}$.
For the case of  noncommutative dyadic martingales, i.e,  when  the von Neumann algebra is  the  hyperfinite type ${\rm II}_1$ factor $\R$ equipped with its natural  increasing filtration,  these  fractional  integrals turn out to be 
 exactly as in \eqref{fractional} (we refer the reader to Section~\ref{section:fractional} below for  details).

\medskip

The paper is organized as follows. In the next section, we collect notions and notation 
from noncommutative symmetric spaces and noncommutative martingale theory necessary for our presentation. In Section~\ref{section:fractional}, we formulate  
the general fractional integrals  and provide systematic studies of their actions on various spaces. In particular, we prove
  results that  mirror those  from  classical settings. Our first result can be  roughly stated as  fractional integrals of order $\alpha$ being    of weak-type $(1, 1/(1-\alpha))$.  Using duality and  interpolations,   we  also obtain boundedness between  various noncommutative Lorentz spaces. Moreover,  they can be strengthened using the noncommutative maximal functions developed by Junge in \cite{Ju}  (see Theorem~\ref{theorem:weak2} and Theorem~\ref{main:maximal}).  These results go beyond Theorem~\ref{classical} in two ways,  they provide a unified approach  to fractional integrals that are not restricted to dyadic martingales and also the  method we use is  general enough to include martingales that are not necessarily regular. 
We  also investigate fractional integrals acting between   noncommutative Hardy spaces. More precisely, 
we obtain a $\mathcal{H}_1$-$\mathcal{H}_p$ boundedness of the fractional integral $I^\alpha$ where $p=1/(1-\alpha)$. This is formulated in  Theorem~\ref{Hardy-spaces-version} below.
 In the last section, we explore when the various results obtained in the previous section can be extended to include the case $0<p<1$. This was accomplished  through the use of noncommutative  atomic decompositions and noncommutative atomic Hardy spaces for martingales. 

\section{Preliminaries and notation}


In this preliminary section we introduce some basic definitions and well-known results concerning  noncommutative $L_p$-spaces and noncommutative martingales. 
 We use standard notation for operator algebras  as may be found in the books \cite{KR,TAK}.


\subsection{Noncommutative symmetric spaces}

In this subsection we will review the general  construction of noncommutative spaces.
Let $\M$ be  a  semifinite von Neumann algebra  equipped with a distinguished   faithful normal semifinite trace $\T$. Assume that $\mathcal{M}$ is acting on a Hilbert space $H$. A closed densely defined operator $x$ on $H$ is said to be  affiliated with $\mathcal{M}$ if $x$ commutes with every unitary $u$ in the commutant $\mathcal{M}'$ of $\mathcal{M}$. If $a$ is a densely defined self-adjoint operator on $H$ and $a=\int_{\mathbb{R}}s de_s^a$ is its spectral decomposition, then for any Borel subset $B\subseteq \mathbb{R}$, we denote by $\chi_B(a)$ the corresponding spectral projection $\int_{\mathbb{R}}\chi_B(s)de_s^a$. An operator $x$ affiliated with $\mathcal{M}$ is called $\tau$-measurable if there exists $s>0$ such that
$\tau(\chi_{(s,\infty)}(|x|))<\infty$.

Let $\wt{\M}$ denote the topological $*$-algebra of all $\tau$-measurable operators. For $x\in \wt{\M}$, 
$$\mu_t(x)=\inf\{s>0:\tau\big(\chi_{(s,\infty)}(|x|)\big)\leq t\},\quad t>0.$$
The function $t \mapsto \mu_t(x)$ from  the interval $[0, \T({\bf
1}))$ to $[0, \infty]$ is called the {\it generalized singular
value function} of $x$. Note that $\mu_t(x)<\infty$ for all $t>0$ and $t \mapsto \mu_t(x)$  is a decreasing function.
   We observe  that if $\M=L^\infty(\real_+)$ then $\widetilde{\M}$ is the space of Lebesgue measurable functions on $\real_+$ and for any given $f \in \widetilde{\M}$, 
$\mu(f)$ is precisely the classical decreasing rearrangement of
the function $|f|$ commonly used in theory of rearrangement invariant  function spaces as described in \cite{BENSHA,LT}.  We refer the reader to \cite{FK}  for a more  in depth
study of $\mu(\cdot)$.

For  $0<p<\infty$, we recall that the noncommutative $L_p$-space associated with $(\mathcal{M},\tau)$ is defined by
$L_p(\mathcal{M},\T)= \{x\in \wt{\M}:\tau(|x|^p)<\infty\}$
with
$$\|x\|_p= \tau(|x|^p)^{1/p}=\Big(\int_0^\infty \mu_t(x)^p dt\Big)^{1/p}.$$

More generally, one can extend the preceding definition to  more general function spaces which we now summarize. We recall  first some basic definitions from
general theory of rearrangement invariant spaces. We denote by
$L_0(\real+)$ the space of all $\mathbb{C}$-valued Lebesgue measurable
functions defined on $\real_+$. 

A  quasi-Banach space $(E,\|\cdot\|_E)$,
where $E\subset L_0(\real_+)$, is called  a \emph{ rearrangement invariant
quasi-Banach function space} if it follows from $f \in E$, $g\in
L^0(\real_+)$, and $\mu(g)\leq \mu(f)$ that $g\in E$ and $\|g\|_E \leq
\|f\|_E$. Furthermore, $(E,\|\cdot\|_E)$ is called \emph{symmetric
Banach function space} if it satisfies  the additional property that $f,g
\in E$ and $g\prec\prec f$ imply that $\|g\|_E \leq \|f\|_E$.  Here $g\prec\prec f$
denotes the submajorization in the sense of
 Hardy-Littlewood-Polya :
\[
\int^t_0 \mu_s(g) \ ds \leq  \int^t_0 \mu_s(f) \ ds, \quad {\rm for \
all}\ t>0.
\]
We refer the reader to \cite{LT} for any unexplained
terminology from the general theory of rearrangement invariant
function spaces and symmetric spaces. 
Given  a semifinite von Neumann algebra $(\M,\T)$ and a
symmetric quasi-Banach function space  $(E, \|\cdot\|_E)$ on  the
interval $[0, \infty)$, we define the  corresponding  noncommutative space by setting:
\begin{equation*}
E(\M, \T) = \big\{ x \in
\widetilde{\M}\ : \ \mu(x) \in E \big\}. 
\end{equation*}
Equipped with the  quasi-norm
$\|x\|_{E(\M,\T)} := \| \mu(x)\|_E$, the space  $E(\M,\T)$ (or simply $E(\M)$) is a complex  quasi-Banach space and is  generally referred to as the \emph{non-commutative symmetric space} associated with $(\M,\T)$ corresponding to $(E, \|\cdot\|_E)$. Extensive discussions on various properties of such spaces can be found in \cite{CS,DDP1,X}.

In this article, we will be mainly working with Lorentz spaces. 
For  $0< p, q \leq \infty$, we  recall the  Lorentz space $L_{p,q}$  as the subspace of all $f \in L_0(\real_+)$ such that
\[
\|f\|_{p,q}=\begin{cases}
\Big(\displaystyle{\int_0^\infty\big(t^{1/p}\mu_t(f)\big)^q \frac{dt}{t}\Big)^{1/q}}, \quad  &{\text{if}}\  0<q<\infty \\
 \displaystyle{\sup_{t>0}t^{1/p}\mu_t(f)},\quad  &{\text{if}}\ q=\infty.
 \end{cases}
 \]
is finite. Clearly, $L_{p,p}(\real_+)=L_p(\real_+)$. If $1\leq q\leq p<\infty$ or $p=q=\infty$, then $L_{p,q}(\real_+)$ is a  symmetric Banach function space. If $1<p<\infty$ and $p\leq q\leq \infty$ then $L_{p,q}(\real_+)$  can be equivalently renormed to  become a  symmetric Banach function space. In general, $L_{p,q}(\real_+)$ is only a symmetric quasi-Banach  function space. Basic properties of Lorentz spaces may be found in \cite{BENSHA,LT}.  Through  the general construction of noncommutative spaces described above  we may define the noncommutative Lorentz space $L_{p,q}(\M,\T)$ associated with  $(\M,\T)$ corresponding to $L_{p,q}(\real_+)$.

We now review some properties of noncommutative Lorentz spaces that we will need throughout. In  the sequel,   we will  make use of the well-known fact that for $1\leq p <\infty$ and $x\in L_{p,\infty}(\M,\T)$ then 
 \[
 \|x\|_{p,\infty}=\sup_{\lambda >0} \lambda(\T(\chi_{(\lambda,\infty)}(|x|)))^{1/p}.
 \]
 The following quasi-triangle inequality is a very simple but useful fact. We refer to \cite{Ran15} for a short proof.
\begin{lemma}\label{triangle-ineq}
Given two operators $x_1$, $x_2$ in $L_{1,\infty}(\mathcal{M},\T)$ and $\lambda>0$, we have
$$\lambda\tau\big(\chi_{(\lambda,\infty)}(|x_1+x_2|)\big)\leq 2\lambda\tau\big(\chi_{(\lambda/2,\infty)}(|x_1|)\big)+2\lambda\tau\big(\chi_{(\lambda/2,\infty)}(|x_2|)\big).$$
\end{lemma}

 From the general duality theory for noncommutative spaces developed by  Dodds {\it et al.} in \cite{DDP3}, we may also state that for 
 $1<p, q<\infty$, 
 \begin{equation}\label{duality}
 \big(L_{p,q}(\mathcal{M},\T)\big)^* = L_{p',q'}(\mathcal{M},\T),
 \end{equation}
 where $p'$ and $q'$ denote the conjugate indices of $p$ and $q$ respectively. 
Noncommutative Lorentz spaces behave well with respect to real
interpolations. Indeed, we may deduce from \cite[Theorem~5.3.1, p.~113]{BL}  and \cite[Corollary 2.2]{PX3}  that  if  $0<\theta<1$, $0<p_j,q_j\leq \infty$ for $j\in\{0,1\}$,  and $p_0 \neq p_1$, then
\begin{equation}\label{interpolation}
L_{p,q}(\mathcal{M},\T)=[L_{p_0,q_0}(\mathcal{M},\T),L_{p_1,q_1}(\M,\T)]_{\theta,q}
\end{equation}
(with equivalent quasi-norms), where $1/p=(1-\theta)/p_0 + \theta/p_1$. All these basic facts will be used in the sequel.


\subsection{Noncommutative  martingales}

In this subsection, we  recall some backgrounds for the theory of  noncommutative martingales. Let $({\mathcal{M}}_{n})_{n\geq 1}$ be an increasing sequence of von Neumann subalgebras of a von Neumann algebra ${\mathcal{M}}$ such that the union of  the ${\mathcal{M}}_{n}$'s  is $w^{*}$-dense in ${\mathcal{M}}$. Assume that there exists a conditional expectation  ${\mathcal{E}}_{n}$ from ${\mathcal{M}}$ onto  ${\mathcal{M}}_{n}$ (this is always the case if $\M$ is a finite von Neumann algebra). It is well-known that ${\mathcal{E}}_{n}$ extends to a bounded projection from $L_1(\mathcal{M}) +\M$ onto $L_1(\mathcal{M}_n) +\M_n$ and consequently, by interpolations, from $L_p(\M)$ onto $L_p(\M_n)$ for all $1\leq p\leq \infty$.

A sequence $x=(x_n)_{n\geq 1}$ in $L_1(\mathcal{M})$ is called a \emph{noncommutative martingale} with respect to $({\mathcal{M}}_{n})_{n\geq 1}$ if
$${\mathcal{E}}_{n}(x_{n+1})=x_n,\quad \forall n\geq 1.$$
If additionally, $x\subseteq L_{p}(\mathcal{M})$ for some $1< p\leq \infty$ then $x$ is called an $L_{p}(\mathcal{M})$-martingale. In this case, we set
$$\|x\|_{p}=\sup_{n\geq 1}\|x_n\|_{p}.$$
If $\|x\|_{p}<\infty$, then $x$ is called a $L_{p}$-bounded martingale. Similarly, we may also consider martingales that are bounded in $L_{p,q}(\M)$  when $1<p\leq \infty$, $0<q\leq \infty$ and set 
$$\|x\|_{p,q}=\sup_{n\geq 1}\|x_n\|_{p,q}.$$
We refer to  \cite{Jiao} for more information on $L_{p,q}$-bounded martingales.

For a given   martingale $x=(x_n)_{n\geq 1}$, we assume the  usual convention that $x_0=0$. The  martingale difference sequence $dx=(dx_k)_{k\geq1}$ of $x$ is defined by  $$dx_k=x_k-x_{k-1}, \quad k\geq 1.$$

Let us now recall the definitions of the square functions and Hardy spaces for noncommutative martingales. 
Following \cite{PX}, we introduce the column and row versions of square functions 
relative to a  martingale $x = (x_n)_{n\geq 1}$:
$$S_{c,n} (x) = \Big ( \sum^n_{k = 1} |dx_k |^2 \Big )^{1/2}, \quad 
S_c (x) = \Big ( \sum^{\infty}_{k = 1} |dx_k |^2 \Big )^{1/2};$$
and
$$S_{r,n} (x) = \Big ( \sum^n_{k = 1} | dx^*_k |^2 \Big )^{1/2}, \quad
S_r (x) = \Big ( \sum^{\infty}_{k = 1} | dx^*_k |^2 \Big)^{1/2}.$$
Let $0 \leq p \leq  \infty$. 
Define  the space $\mathcal{H}_p^c (\mathcal{M})$
(resp. $\mathcal{H}_p^r (\mathcal{M})$) as the completion of all
finite martingales in $\M \cap L_p(\M)$ under the (quasi) norm $\| x \|_{\mathcal{H}_p^c}=\| S_c (x) \|_p$
(resp. $\| x \|_{\mathcal{H}_p^r}=\| S_r (x) \|_p $).  When $1\leq p \leq  \infty$,  $\mathcal{H}_p^c (\mathcal{M})$ and $\mathcal{H}_p^r (\mathcal{M})$ are Banach spaces while for $0<p<1$, they are only $p$-Banach spaces. 

The Hardy space of noncommutative martingales is defined as
follows: if $0 \leq p < 2,$
\begin{equation*}
\mathcal{H}_p(\mathcal{M}) 
= \mathcal{H}_p^c (\mathcal{M}) + \mathcal{H}_p^r(\mathcal{M})
\end{equation*}
equipped with the (quasi) norm
\begin{equation*}
\| x \|_{\mathcal{H}_p} = 
\inf \big \{ \| y\|_{\mathcal{H}_p^c} + \| z \|_{\mathcal{H}_p^r} \big\},
\end{equation*}
where the infimum is taken over all 
$y \in\mathcal{H}_p^c (\mathcal{M} )$ and $z \in \mathcal{H}_p^r(\mathcal{M} )$ 
such that $x = y + z.$ 
For $2 \leq p \leq \infty,$
\begin{equation*}
\mathcal{H}_p (\mathcal{M}) =
\mathcal{H}_p^c (\mathcal{M}) \cap \mathcal{H}_p^r(\mathcal{M})
\end{equation*}
equipped with the norm
\begin{equation*}
\| x \|_{\mathcal{H}_p} = 
\max \big \{ \| x\|_{\mathcal{H}_p^c} , \| x \|_{\mathcal{H}_p^r} \big\}.
\end{equation*}

We also need $\ell_p(L_p(\M))$, the space of all sequences $a=(a_n)_{n\geq 1}$ in $L_p(\M)$ such that
$$\|a\|_{\ell_p(L_p(\M))}=\Big(\sum_{n\geq 1}\|a_n\|_p^p\Big)^{1/p} <\infty.$$
Set
$$s_d(x)=\Big(\sum_{n\geq 1}|dx_n|_p^p\Big)^{1/p}.$$
We note that 
$$\|s_d(x)\|_p=\|dx\|_{\ell_p(L_p(\M))}.$$
Let $\h_p^d(\M)$ be the subspace of $\ell_p(L_p(\M))$ consisting of all martingale difference sequences.  We also would like to mention  that there are other Hardy spaces such as the noncommutative conditioned Hardy spaces in the literature but will not be used in this paper. 

\smallskip

Our primary examples are    noncommutative martingales 
in various Lorentz spaces associated with the type ${\rm II}_1$-hyperfinite factor $\mathcal{R}$. 
Let $\mathbb{M}_2$ be the algebra of $2\times 2 $ matrices with the usual normalized trace $\tr_2$.  Recall that
$$(\mathcal{R},\tau)= \overline{\bigotimes_{i\geq 1}}(\mathbb{M}_2,\tr_2).$$
For $n\geq 1$, we denote by $\mathcal{R}_n$ the finite dimensional von Neumann subalgebra  given by the finite tensor product $\bigotimes_{1\leq i\leq n}(\mathbb{M}_2,\tr_2)$ of $\mathcal{R}$.
 It is customary to identify   $\mathcal{R}_n$  with   $\mathbb{M}_{2^n}$, where $\mathbb{M}_{2^n}$ is the algebra of $2^n \times 2^n$ matrices equipped with the normalized trace $\tr_{2^n}$. Moreover, we view $\mathcal{R}_n$ as a von Neumann subalgebra of $\mathcal{R}_{n+1}$ via the inclusion
$$x \in \mathcal{R}_n \longmapsto x\otimes \textbf{1}_{\mathbb{M}_2} =
\begin{pmatrix}
x & 0\\
0 & x\\
\end{pmatrix}
 \in   \mathcal{R}_{n+1},
$$
where $\textbf{1}_{\mathbb{M}_2}$ is the identity of $\mathbb{M}_2$. With these inclusions, it is clear that  $(\mathcal{R}_n)_{n\geq 1}$ forms an increasing filtration of von Neumann subalgebras  whose union is weak*-dense in $\R$.
Martingales corresponding to the filtration $(\mathcal{R}_n)_{n\geq 1}$ are called \emph{\lq\lq noncommutative" dyadic martingales}.  They are indeed generalizations of dyadic martingales from classical probability theory.

We conclude this subsection with the statement of the noncommutative  Gundy's decomposition from \cite{PR} which will be very crucial in the sequel. Below, $\mathrm{supp}(a)$ denotes the support projection of the measurable operator $a$ in the sense of \cite{TAK}.

\begin{theorem}[\cite{PR}]\label{Gundy} If $x=(x_n)_{n \ge 1}$ is a $L_1$-bounded
noncommutative martingale and $\lambda$ is a positive real
number, there exist four martingales $\varphi$, $\psi$, $\eta$,
and $\upsilon$ satisfying the following properties for some
absolute constant $\mathrm{c}$:
\begin{itemize}
\item[(i)] $x=\varphi +\psi + \eta + \upsilon$;
\item[(ii)] the martingale $\varphi$ satisfies $$\|\varphi\|_1 \leq
\mathrm{c} \|x\|_1, \quad \|\varphi\|_2^2 \leq \mathrm{c}
\lambda\|x\|_1, \quad \|\varphi\|_\infty \leq \mathrm{c} \lambda;$$
\item[(iii)] the martingale $\psi$ satisfies $$\sum_{k=1}^{\infty}
\|d\psi_k\|_1 \le \mathrm{c} \|x\|_1;$$
\item[(iv)] $\gamma$ and $\upsilon$ are $L_1$-martingales with
$$\max \Big\{ \lambda \tau \Big( \bigvee_{k \ge 1} \mathrm{supp}
|d\eta_k| \Big), \, \lambda \tau \Big( \bigvee_{k \ge 1}
\mathrm{supp} \, |d\upsilon_k^*| \Big) \Big\} \le \mathrm{c}
\|x\|_1.$$
\end{itemize}
\end{theorem}


In the sequel, letters $C_p, \kappa_p, \dots$ will denote positive constants depending only on the involved subscripts, and $C, \kappa, \dots$ are absolute constants. All these constants can change from lines to lines.



\section{Noncommutative fractional integrals}\label{section:fractional}

In this section, we define  fractional integrals for noncommutative  martingales.  For the remaining of the paper, we assume that $\M$ is a hyperfinite and finite von Neumann algebra   and the filtration $(\M_k)_{k\geq 1}$ consists of finite dimensional von Neumann subalgebras of $\M$. 

 Fix $k\geq 1$,  we define the difference operator $\mathcal{D}_k=\E_k -\E_{k-1}$ where $\E_0=0$.  Let
 \[
 \mathcal{D}_{k,p}:=\mathcal{D}_k(L_p(\M))=\big\{x \in L_p(\M_k); \E_{k-1}(x)=0\big\}.
 \]
 
  Since $\dim(\M_k)<\infty$,   the $\mathcal{D}_{k,p}$'s   are finite dimensional  subspaces of $L_p(\M)$ for all $1\leq p\leq \infty$. Moreover,    for $p\neq q$,  the two spaces  $\mathcal{D}_{n,p}$ and $\mathcal{D}_{n,q}$ coincide as sets. In particular,
 the formal identity $\iota_k: \mathcal{D}_{k,\infty} \to \mathcal{D}_{k,2}$  forms a natural  isomorphism between the two spaces.
 
 For $k\geq 1$,  set
\begin{equation}\label{zeta}
\zeta_k:=1/ \| \iota_k^{-1}\|^{2}.
\end{equation}
Clearly, $0<\zeta_k \leq 1$ for all $k\geq 1$ and   $\lim_{k\to \infty} \zeta_k=0$. Moreover,  for every  $x \in \mathcal{D}_{k,2}$, we have
\begin{equation}\label{infty-2}
\|x\|_\infty \leq \zeta_k^{-1/2} \|x\|_2. 
\end{equation}
Furthermore, if we denote by $j_k$  the inclusion map from $\mathcal{D}_{k,\infty}$ into $\M_k$, then one can easily verify that for every $x \in L_1(\M_k)$, 
$(j_k \iota_k^{-1} \mathcal{D}_k)^*(x)=\E_k(x)-\E_{k-1}(x)  \in L_2(\M_k)$ (here $\mathcal{D}_k: L_2(\M_k) \to \mathcal{D}_{k,2}$). In particular,  for every $x\in \mathcal{D}_{k,1}$,
\begin{equation}\label{2-1}
\|x\|_2 \leq 2\zeta_{k}^{-1/2} \|x\|_1.
\end{equation}

The following definition constitutes the main topic of this paper. This was primarily inspired by  a similar notion used by Chao and Ombe \cite{Chao-Ombe} for classical dyadic martingales  in $L_1[0,1]$ described in the introduction.  We propose a setup that goes beyond dyadic situation.
\begin{definition}\label{definition:fractional}
For a given noncommutative  martingale $x=(x_n)_{n\geq 1}$ and $0<\alpha<1$, we define the \emph{fractional 
integral of order $\alpha$}  of $x$ to be the sequence $I^\alpha x=\{(I^\alpha x)_n\}_{n\geq 1}$ where  for every $n\geq 1$,
$$(I^\alpha x)_n = \sum_{k=1}^{n} \zeta_k^{\alpha}dx_k$$
with the sequence of scalars $(\zeta_k)_{k\geq 1}$ from \eqref{zeta}.
\end{definition}

Since $\alpha>0$,  the operation $I^\alpha $ is  a martingale transform with bounded coefficients  and thus, according to  \cite{PX,Ran15},  $I^\alpha$ is of strong type $(p,p)$ for $1<p<\infty$ and  is of weak type $(1,1)$.  In particular, if $x$ is a  $L_1$-bounded  martingale then $\{(I^\alpha x)_n\}_{n\geq 1}$ is a  martingale (adapted to the same filtration) that is bounded  in $L_{1,\infty}(\M)$.

We will provide a short discuss at the  end of this section about  the reason that motivates our choice of  the  scalar coefficients  $(\zeta_k)_{k\geq 1}$ as  defined in \eqref{zeta} and  point out that it  is the optimal choice for all the results  in this section to hold. 
We should  also emphasize here that for the case of "noncommutative" dyadic filtration on $\R$, one can  easily verify that 
$\zeta_k=2^{-k}$ for $k\geq 1$ and therefore  our definition is indeed  a proper   generalization of the classical  dyadic fractional integrals described in the introduction.

\medskip

Our goal is  to  explore  strengthening of the above stated facts about martngale transforms. 
   More precisely, we aim to generalize Theorem~\ref{classical} to our abstract noncommutative settings. In particular,  we obtain that $I^\alpha$ is of weak type $\big(1,1/(1-\alpha)\big)$.  This specific  result  leads to various weak-type inequalities and boundedness of  fractional integrals between  different Lorentz spaces. 
   
\subsection{Weak-type boundedness and consequences} \label{weak-type}

The following weak-type estimate  is the main result of this subsection.

\begin{theorem}\label{theorem:weak2}
Let $0<\alpha<1$. There exists a constant ${\mathrm c}_\alpha$ such that if  
 $x$ is a $L_1$-bounded dyadic martingale then  
\[
\big\|I^\alpha x\big\|_{L_{1/(1-\alpha),\infty}(\mathcal{M})}\leq {\mathrm c}_\alpha  \big\|x \big\|_{1}.
\]
\end{theorem}

In preparation for  the proof of Theorem~\ref{theorem:weak2}, we  establish  first various  preliminary lemmas. 

\begin{lemma}\label{lemma:basic} Let $k\geq 1$ and $a\in \mathcal{D}_{k, \infty}$. Then
\begin{enumerate}[\rm(i)]
\item For any given $0<\alpha<1$,
$$
\zeta_k^\alpha \|a\|_{1/(1-\alpha)} \leq 2^\alpha \|a\|_1.$$
\item For $1<p<2$ and $\alpha=1/p -1/2$,
$$
\zeta_k^\alpha \|a\|_2 \leq \|a\|_p. $$
\end{enumerate}
\end{lemma}

\begin{proof} For item $(i)$,  we have: 
\begin{align*}
\zeta_k^\alpha \|a\|_{1/(1-\alpha)} &=\zeta_k^\alpha \T\big(|a|^{1/(1-\alpha)}\big)^{1-\alpha}\\
&=\zeta_k^\alpha \T\big(|a|^{\alpha/(1-\alpha)} |a|\big)^{1-\alpha}\\
&\leq \zeta_k^\alpha \|a\|_\infty^\alpha \|a\|_1^{1-\alpha}.
\end{align*}
By \eqref{infty-2} and \eqref{2-1}, we have
\[
\|a\|_\infty \leq  \zeta_k^{-1/2} \|a\|_2 \leq 2\zeta_k^{-1} \|a\|_1.
\]
Therefore, when combined with the above estimate, it leads to
\[
\zeta_k^\alpha \|a\|_{1/(1-\alpha)} \leq 2^\alpha\zeta_k^\alpha \zeta_k^{-\alpha} \|a\|_1^\alpha  \|a\|_1^{1-\alpha}=2^\alpha \|a\|_1.
\]
The argument for item $(ii)$ is similar. Assume that $\alpha=1/p-1/2$ and $a \in \mathcal{D}_{k,\infty}$. Then 
\begin{align*}
\|a\|_2 &= \T\big(|a|^2 \big)^{1/2}\\
&=\T\big(|a|^{2-p} |a|^p)^{1/2}\\
&\leq  \|a\|_\infty^{(2-p)/2} \|a\|_p^{p/2}\\
&\leq  (\zeta_k^{-1/2})^{(2-p)/2} \|a\|_2^{(2-p)/2} \|a\|_p^{p/2}\\
&\leq \zeta_k^{-\alpha p/2} \|a\|_2^{(2-p)/2} \|a\|_p^{p/2}.
\end{align*}
This implies that $\zeta_k^{\alpha p/2} \|a\|_2^{p/2} \leq \|a\|_p^{p/2}$ which after raising   to the power $2/p$  gives the stated inequality.
\end{proof}

As immediate consequences of Lemma~\ref{lemma:basic}, we obtain
\begin{lemma}\label{step1} \begin{enumerate}[{\rm(i)}] 
\item For $0<\alpha<1$, $I^\alpha$ is bounded from $\h_1^d(\M)$ into $L_{1/(1-\alpha)}(\M)$.
\item If $1<p<2$ and  $\alpha_0=1/p-1/2$, then there exist a constant ${\mathrm c}_p$ so that for every $z \in L_p(\M)$,
\begin{equation*}
\big\|(I^{\alpha_0} z)_n \big\|_2\leq {\mathrm c}_p\big\|z\big\|_p.
\end{equation*}
That is, $I^{\alpha_0}$ is bounded from $L_p(\M)$ into $L_2(\M)$.
\end{enumerate}
\end{lemma}
\begin{proof} The first item is immediate from Lemma~\ref{lemma:basic}~(i). For the second item, fix $z \in L_p(\M)$ and $n\geq 1$. Then,  since for every $k\geq 1$,   $dz_k \in \mathcal{D}_{k,\infty}$, we may deduce  from Lemma~\ref{lemma:basic}~(ii) that
\begin{align*}
\|(I^{\alpha_0} z)_n\|_2^2  &=\sum_{k=1}^n \zeta_k^{2\alpha_0}\|dz_k\|_2^2 \\
&\leq \sum_{k=1}^n \|dz_k\|_p^2.
\end{align*}
Using the fact that $L_p(\M)$ is of cotype 2 (\cite{PX3}), it follows that  there is a constant $\kappa_p$ such that 
\[
\|(I^{\alpha_0} z)_n\|_2^2  \leq \kappa_p^2 \mathbb{E} \big\|\sum_{k=1}^n \varepsilon_k dz_k \big\|_p^2 
\]
where $(\varepsilon_k)_k$ is a Rademacher sequence and $\mathbb{E}$ denotes the expectation on the $\varepsilon_k$'s. Furthermore, by  the $L_p$-boundedness of martingale 
transforms (see \cite{PX}), there is another constant $\beta_p$ so that 
\[
\|(I^{\alpha_0} z)_n\|_2^2 \leq \kappa_p^2\beta_p^2 \|z\|_p^2,
\]
which proves $(ii)$.
\end{proof}

\begin{lemma}\label{step2}
Let $1<p<2$, $1/p +1/{p'}=1$, and  $\alpha=1/p -1/{p'}$. Then $I^\alpha$ is bounded from $L_p(\M)$ into $L_{p'}(\M)$. More precisely, 
  for every $z \in L_p(\M)$ and $n\geq 1$,
\begin{equation*}
\| (I^\alpha z)_n\|_{p'} \leq {\mathrm c}_p^2\|z\|_p, 
\end{equation*}
where ${\mathrm c}_p$ is the constant from Lemma~\ref{lemma:basic}~(ii).
\end{lemma}
\begin{proof}
Note first that $\alpha=2\alpha_0$ where $\alpha_0$ is from Lemma~\ref{step1}~(ii). Fix $y\in L_p(\M)$ with $\|y\|_p=1$. Then we have
\begin{align*}
\big| \langle I^\alpha z, y\rangle\big| &=\big| \T\big( (I^\alpha z)y^*\big)\big| \\
&=\left| \T\left( \Big(\sum_k \zeta_k^{\alpha } dz_k\Big)\Big(\sum_k dy_k^* \Big)\right) \right|\\
&=\left| \T\Big(\sum_k \zeta_k^{\alpha } dz_k dy_k^*\Big)\right| \\
&=\left| \T\Big(\sum_k \zeta_k^{2\alpha_0 } dz_k dy_k^*\Big)\right| \\
&=\left| \T\left( \Big(\sum_k \zeta_k^{\alpha_0 } dz_k\Big)\Big(\sum_k  \zeta_k^{\alpha_0 }dy_k^* \Big)\right) \right| \\
&\leq \left\| I^{\alpha_0}z \right\|_2 \left\| I^{\alpha_0}y \right\|_2.
\end{align*} 
It then follows from Lemma~\ref{step1}~(ii) that $\big| \langle I^\alpha z,  y\rangle\big| \leq {\mathrm c}_p^2 \big\|z\big\|_p$. Since $y$ is arbitrary, the desired inequality follows.
\end{proof}

\begin{proof}[Proof of Theorem~\ref{theorem:weak2}] We have   to prove  the existence of  a constant ${\mathrm c}_\alpha$  such that for  any fixed $n \geq 1$ and  every $s>0$, we have
\begin{equation}\label{equiv:weaktype}
\tau\Big(\chi_{(s,\infty)}\big(\big|(I^\alpha x)_n\big|\big)\Big)\leq {\mathrm c}_\alpha \left(\frac{\|x\|_1}{s}\right)^{1/ (1-\alpha)}.
\end{equation}
By linearity and homogeneity, we may assume without loss of generality that  $x\geq 0$ with $\|x\|_1=1$. 
Since the trace $\T$ is normalized, it is enough to consider only the case  $s>1$. Let $\lambda=s^{1/(1-\alpha)}$. 

We apply the noncommutative Gundy's decomposition stated in Theorem~\ref{Gundy} to the martingale $x$ and $\lambda>1$. There exist four martingales $\varphi$, $\psi$, $\eta$, and $ \upsilon$ with
$x=\varphi+\psi + \eta+ \upsilon$ and satisfy the properties enumerated  in Theorem~\ref{Gundy}.

Clearly, we have
 for any  given $n\geq 1$,
$$(I^\alpha x)_n = (I^\alpha \varphi)_n +(I^\alpha \psi)_n + (I^\alpha \eta)_n + (I^\alpha \upsilon)_n.$$

Using the elementary inequality $|a+b|^2 \leq 2|a|^2 + 2|b|^2$ for operators, we have 
\[
\big|(I^\alpha x)_n \big|^2 \leq 4 \big|(I^\alpha \varphi)_n\big|^2 + 4\big|(I^\alpha \psi)_n \big|^2 + 4\big|(I^\alpha \eta)_n \big|^2 +  4\big|(I^\alpha \upsilon)_n\big|^2.
\]

Now,   according to Lemma~\ref{triangle-ineq}, we have
\begin{align*}
\T\Big(\chi_{(s,\infty)}\big(\big|(I^\alpha x)_n\big|\big)\Big) &=\T\Big(\chi_{(s^2,\infty)}\big(\big|(I^\alpha x)_n\big|^2\big)\Big)\\
&\leq  4 \T\Big(\chi_{(s^2/4,\infty)}\big(4\big|(I^\alpha \varphi)_n\big|^2\big)\Big) +
4 \T\Big(\chi_{(s^2/4,\infty)}\big(4\big|(I^\alpha \psi)_n\big|^2\big)\Big)  \\
&\ + 4 \T\Big(\chi_{(s^2/4,\infty)}\big(4\big|(I^\alpha \eta)_n\big|^2\big)\Big) +
4 \T\Big(\chi_{(s^2/4,\infty)}\big(4\big|(I^\alpha \upsilon)_n\big|^2\big)\Big) \\
&:= I + II +III + IV.
\end{align*}
It suffices to estimate $I$, $II$, $III$, and $IV$ separately.

\medskip

For $I$,  fix $1<p<2$,  $1/p +1/{p'}=1$ so that   $\alpha=1/p -1/{p'}$.  Then  using Chebychev's inequality and  
the result already established in Lemma~\ref{step2} that $I^\alpha$ is bounded from $L_p(\M)$ into $L_{p'}(\M)$, we get 
\begin{align*}
I&=  4 \T\Big(\chi_{(s/2,\infty)}\big(2\big|(I^\alpha \varphi)_n\big|\big)\Big)\\
 &\leq  4^{p' +1}s^{-p'}\big\|(I^\alpha \varphi)_n\big\|_{p'}^{p'} \\
&\leq 4^{p' +1} {\mathrm c}_p^2 s^{-p'} \big\|\varphi\big\|_{p}^{p'} \\
&= 4^{p' +1}{\mathrm c}_p^2 s^{-p'} \tau\big(|\varphi|^p\big)^{p'/p}\\
&=4^{p' +1} {\mathrm c}_p^2 s^{-p'} \tau\big(|\varphi|^{p-1}|\varphi|\big)^{p'/p}\\
&\leq 4^{p' +1} {\mathrm c}_p^2 s^{-p'} \big\|\varphi\big\|_\infty^{(p-1)p'/p} \big\|\varphi \big\|_1^{p'/p}.
\end{align*}
Since $ \|\varphi\|_1\leq {\mathrm c}$ and $ \|\varphi\|_\infty \leq {\mathrm c}\lambda$, we deduce that
 \[
I \leq  4^{p' +1} {\mathrm c}_p^2  {\mathrm c}^{p'} s^{-p'} \lambda^{(p-1)p'/p}  
\leq  4^{(2p-1)/(p-1)}{\mathrm c}_p^2{\mathrm c}^{p/(p-1)} \lambda^{-1}
\]
which shows the existence of a constant  $\mathrm c_\alpha$ so that 
\[
I\leq  {\mathrm c}_\alpha \lambda^{-1}.
\]
 
 For $II$, we first apply Chebychev's inequality as above to get
 \begin{align*}
 II &=4 \T\Big(\chi_{(s/2,\infty)}\big(2\big|(I^\alpha \psi)_n\big|\big)\Big)\\
  &\leq  4^{(2-\alpha)/(1-\alpha)} s^{-1/(1-\alpha)} \big\| (I^\alpha \psi)_n \big\|_{1/(1-\alpha)}^{1/(1-\alpha)}.
 \end{align*}
According to Lemma~\ref{step1}~(i), $I^\alpha$ is bounded from $\h_1^d(\M)$ into $L_{1/(1-\alpha)}(\M)$.  Therefore, we may deduce that 
\begin{equation}\label{chebychev}
 II \leq {\mathrm c}_\alpha \lambda^{-1} \big\| \psi \big\|_{ \h_1^d}^{1/(1-\alpha)}
\end{equation}
for ${\mathrm c}_\alpha=4^{(2-\alpha)/(1-\alpha)}$. Combining  \eqref{chebychev} with  Theorem~\ref{Gundy}~(iii) provides the desired estimate for $II$.

\medskip

To estimate $III$, we note that  using polar decompositions of  the $d\eta_k$'s, the operator $(I^\alpha \eta)_n$ is right-supported by the projection 
$\bigvee_{k\geq 1}\mathrm{supp}
|d\eta_k| $. Consequently,   the operator  $| (I^\alpha \eta)_n | $ is supported by $\bigvee_{k\geq 1}\mathrm{supp}
|d\eta_k|$ and thus  we may conclude from Theorem~\ref{Gundy}~(iv) that 
\[
III\leq 4\T\Big(\bigvee_{k\geq 1}\mathrm{supp} 
|d\eta_k|\Big) \leq 4{\mathrm c} \lambda^{-1}.
\]

\medskip

For the last item $IV$, we observe that $(I^\alpha \upsilon)^* =I^\alpha \upsilon^*$. Arguing as in the case of $III$, we have that   $| (I^\alpha \upsilon)_n^* | $ is supported by  the projection $\bigvee_{k\geq 1}\mathrm{supp}
|d\upsilon_k|$.  Similarly,  we may deduce from  Theorem~\ref{Gundy}~(iv) that
\[
IV=4\T \Big(\chi_{(s/2,\infty)}\big(2\big|(I^\alpha\upsilon)_n^*\big|\big)\Big)
\leq  4\T\Big(\bigvee_{k\geq 1}\mathrm{supp} 
|d\upsilon_k|\Big) \leq 4 {\mathrm c} \lambda^{-1}.
\]
 
As noted above, combining the estimates on $I$, $II$, $III$, and $IV$ proves 
\eqref{equiv:weaktype}. The proof of Theorem \ref{theorem:weak2} is complete.
\end{proof}


We now consider some applications of Theorem~\ref{theorem:weak2} to strong type boundedness of fractional integrals.
Given $0<\alpha<1$,  we observe from \eqref{duality} that the noncommutative Lorentz space  $L_{{1/\alpha},1}(\M)$ 
is the K\"othe dual of noncommutative symmetric space $L_{{1}/{(1-\alpha)},\infty}(\M)$ in the sense of \cite{DDP3}. Thus, it immediately follows from Theorem~\ref{theorem:weak2} that  restricting the adjoint of $I^\alpha$ to the K\"othe dual  implies that $(I^\alpha)^*: L_{{1}/{\alpha},1}(\M) \to \M$ is bounded. On the other hand, it can be easily verified from the definition that  the adjoint $(I^\alpha)^*$ is formally equal to  the fractional integral $I^\alpha$  itself and thus we may state:

\begin{corollary}\label{adjoint} Let $0<\alpha<1$. Then  $I^\alpha$ is bounded from 
$L_{{1}/{\alpha},1}(\M,\T)$ into $\M$.
\end{corollary}

Using interpolation, we also get:

\begin{corollary}\label{theorem:main} Let $1<p<q<\infty$, $0<r\leq \infty$, and $\alpha=1/p-1/q$.  The  mapping $I^\alpha$ is bounded from $L_{p,r}(\M)$ into $L_{q,r}(\M)$.  In particular, $I^\alpha$ is bounded from $L_p(\M)$ into $L_q(\M)$.
\end{corollary}

 \begin{proof} Interpolating Theorem \ref{theorem:weak2} and Corollary~\ref{adjoint}, we have for $0<\theta<1$ and $0<r \leq \infty$:
 \[
 I^\alpha: [L_{1}(\M),L_{{1}/{\alpha},1}(\M)]_{\theta,r} \longrightarrow [L_{{1}/{(1-\alpha)},\infty}(\M),L_{\infty}(\M)]_{\theta,r}
 \]
is bounded.  Choosing $\theta$ so that  $1/p = 1+(\alpha-1)\theta$ and $1/q=(1-\theta)(1-\alpha)$, the interpolation result stated in \eqref{interpolation}  yields  the desired conclusion.
\end{proof}



Next, we consider  improvements of Theorem~\ref{theorem:weak2} and Corollary~\ref{theorem:main} using maximal functions. For this, let us recall the noncommutative 
$\ell_\infty$-valued  spaces considered  first in \cite{PIS7,Ju} for the noncommutative $L_p$-spaces and in \cite{Dirksen2} for the more general case of noncommutative symmetric spaces.

Let $E$ be a symmetric Banach function space on $\real_+$ and $\N$ be a semi finite von Neumann algebra equipped with a semi finite trace $\sigma$. We set $E(\N,\ell_\infty)$ to be the space of all sequences  $x=(x_k)_{k\geq 1}$ in $E(\N,\sigma)$  for which there exist $a, b \in  E^{(2)}(\N,\sigma)$ and a bounded sequence $y=(y_k)_{k\geq 1}$ in $\N$ such that for every $k\geq 1$,
\[
x_k= ay_kb,
\] 
where  $E^{(2)}(\N,\sigma)=\{a\in \wt{\N}: |a|^2 \in E(\N, \sigma)\}$ equipped with  the norm $\big\|a\big\|_{E^{(2)}(\N)}=\big\||a|^2\big\|_{E(\N)}^{1/2}$.

For $x \in E(\N;\ell_\infty)$, we define
\[
\big\| x \big\|_{E(\N;\ell_\infty)} :=\inf\Big\{ \big\|a\big\|_{E^{(2)}(\N)} \sup_{k\geq 1} \big\|y_k\big\|_\infty \big\| b\big\|_{E^{(2)}(\N)} \Big\},
\]
where the infimum is taken over all possible factorizations of $x$ as described  above. We should point out that in the  case where  $(|x_k|)_{k\geq 1}$ is a commuting sequence and thus the maximal function $Mx=\sup_{k\geq 1}|x_k|$ is well-defined, the value of   $\|(x_k)_{k\geq 1}\|_{E(\N;\ell_\infty)}$ is precisely the norm of 
$Mx$ in $E(\N,\sigma)$.  This justifies the use  of  the space $E(\N;\ell_\infty)$ as a substitute for the lack of supremum or maximum for sets of noncommuting  operators. This remarkable discovery was made by Junge in \cite{Ju} where among other things he  applied this analogy to formulate the noncommutative Doob's maximal inequalities.

Before proceeding, we also need to recall the notion
of Boyd indices. Let $E$ be symmetric  Banach space
on $(0,\infty)$. For $s>0$, the dilation operator $D_s: E \to E$
is defined by setting
\begin{equation*}
D_sf(t)=f(t/s), \qquad t>0, \qquad f \in E.
\end{equation*}
The {\it lower and upper Boyd indices } of $E$ are defined by
\begin{equation*}
p_E :=\lim_{s\to \infty}\frac{\log s}{\log\Vert D_s\Vert} \ \text{and}\ 
q_E :=\lim_{s\to 0^+}\frac{\log s}{\log\Vert D_s\Vert}.
\end{equation*}
It is  a well known fact that $1\leq p_E \leq
q_E \leq \infty$. Moreover, if  $E=L_p$  for $1\leq p \leq
\infty$ then $p_E =q_E=p$.

The key tool  we use is provided by  a recent  generalization  of Junge's noncommutative  Doob's maximal inequality  due to Dirksen which we now state:
\begin{theorem}[{\cite[Corollary~5.4]{Dirksen}}]\label{maximal}
Let  $(\E_n)_{n\geq 1}$  be an increasing sequence of conditional expectations in $(\N,\sigma)$. If $E$ is a symmetric Banach space on $\real_{+}$ with $p_E>1$, then there is a constant ${\mathrm c}_E$  depending only on $E$  such that
\[
\big\| (\E_n(x))_{n\geq 1}\big\|_{E(\N;\ell_\infty)}\leq {\mathrm c}_E \big\|x\big\|_{E(\N)}, \quad x\in E(\N,\sigma).
\]
\end{theorem}
Our next  result follows from combining  Theorem~\ref{theorem:weak2}, Corollary~\ref{theorem:main},  Theorem~\ref{maximal}, and the fact that  if $1<p<\infty$, $1\leq q \leq \infty$ then the upper and lower Boyd indices of $L_{p,q}(\real_+)$ are both equal to the first index $p$ (see for instance \cite[Theorem~4.6]{BENSHA}). It provides   noncommutative generalizations  of   classical results stated in Theorem~\ref{classical}.

\begin{theorem}\label{main:maximal} {\rm1)}  Let $0<\alpha<1$. There exists a constant ${\mathrm c}_\alpha$ such that if  
 $x$ is a $L_1$-bounded noncommutative  martingale then  
\[
\big\|I^\alpha x\big\|_{L_{1/(1-\alpha),\infty}(\M ; \ell_\infty)}\leq  {\mathrm c}_\alpha \big\|x\big\|_{1}.
\]

{\rm 2)}  Let $1<p<q<\infty$, $0<r\leq \infty$, and $\alpha=1/p-1/q$. There exists a constant $\kappa_{\alpha,r}$ such that  if $x$ is a noncommutative  martingale that is bounded in $L_{p,r}(\M)$,  then 
\[
\big\|I^\alpha x\big\|_{L_{q,r}(\M ; \ell_\infty)}\leq \kappa_{\alpha,r}  \big\|x\big\|_{p,r}.
\]
In particular, there exists a constant $\kappa_{\alpha}$ such that  if $x$ is a  martingale that is bounded  in $L_p(\M)$, then 
\[
\big\|I^\alpha x\big\|_{L_{q}(\M ; \ell_\infty)}\leq \kappa_{\alpha}  \big\|x\big\|_{p}.
\]
\end{theorem}

We  conclude this subsection with a note that  all the indices involved in Theorem~\ref{main:maximal} are optimal. In fact, they cannot be improved even for 
    classical  dyadic martingales. We assume that these facts are known but we could not find any specific reference  in the literature.
    For completeness,  we include a simple example to support these claims.  
\begin{example}\label{example} We consider  classical dyadic martingales and fractional integrals as defined in \eqref{fractional}.
Fix $N\geq 1$ and set $f_N :=2^N\chi_{[0,2^{-N})}$.  Then we have the following properties:
\begin{enumerate}[{\rm (i)}]
\item  $\|f_N\|_1=1$.
\item For every $0<\varepsilon <1$, $\displaystyle{\|f_N\|_{(4 -\varepsilon)/3}= 2^{\big(\frac{1-\varepsilon}{4-\varepsilon}\big)N}}$.
\item $\|I^{1/2} f_N\|_2=\sqrt{N/2}$.
\item $\| I^{1/4} f_N\|_2  \sim 2^{N/4}$.
\end{enumerate}
Consequently,  we  may deduce that \[\lim_{N \to \infty} \|I^{1/2}f_N\|_2/\|f_N\|_1=\infty \ \text{and}\  
\lim_{N \to \infty} \|I^{1/4}f_N\|_2/\|f_N\|_{(4-\varepsilon)/3}=\infty.
\]
\end{example}
 
The first  two items can be easily verified. For the  last two items, we note first that for every $1\leq k\leq N$,  $f_k=\mathbb{E}(f_N/ \mathcal{F}_k)=2^k\chi_{[0,2^{-k})}$ and therefore
 $df_k= 2^{k-1}\chi_{[0,2^{-k})} - 2^{k-1}\chi_{[2^{-k}, 2^{-(k-1)})}$.
One can then easily see that $\|df_k\|_2^2=2^{k-1}$  for $1\leq k\leq N$ and thus 
\[
\|I^{1/2}f_N\|_2^2=
\sum_{k=1}^N 2^{-k}\|df_k\|_2^2= N/2.
\]
Moreover, for $I^{1/4}$ we have the following identities
\begin{align*}
\| I^{1/4} f_N \|_2^2 &=\sum_{k=1}^N 2^{-k/2} \|df_k\|_2^2 \\
&=\sum_{k=1}^N 2^{-k/2} 2^{k-1}
=(2^{N/2}-1)/(2-\sqrt{2}).
\end{align*}
This  verifies  the last item. The statements about  the two limits clearly follow  from the listed four items. \qed

 Example~\ref{example}  shows that $I^{1/2}$ is not bounded from $L_1[0,1]$ into $L_2[0,1]$.  In particular,  it confirms  that the weak-type $(1, 1/(1-\alpha))$ boundedness of $I^\alpha$   in Theorem~\ref{theorem:weak2} cannot be improved to strong type.
Moreover,  the index $1/(1-\alpha)$ is the best possible as $I^{1/2}$  cannot be bounded from $L_1[0,1]$ into $L_{q,\infty}[0,1]$ for any $q>2$ since the formal inclusion is bounded from $L_{q,\infty}[0,1]$ into $L_2[0,1]$. On the other hand, Example~\ref{example} also shows that for any  given $0<\varepsilon<1$,
$I^{1/4}$ is not bounded from $L_{(4-\varepsilon)/3}[0,1]$ into $L_2[0,1]$. Taking adjoint and setting $\delta=3\varepsilon/(1-\varepsilon)$, we may state that $I^{1/4}$ is not bounded from $L_2[0,1]$ into $L_{4+\delta}[0,1]$ for any $\delta>0$. This reveals that the indices from Corollary~\ref{theorem:main} are the best possible in the sense that if $\alpha=1/p-1/q$  with $1<p<q <\infty$ then $L_p$-$L_q$ boundedness of $I^\alpha$ is  optimal.



\subsection{Fractional integrals and Hardy spaces}
In this subsection, we will examine  boundedness of fractional integrals with respect to martingale Hardy space norms.  The following theorem is our primary result in this subsection. It may be viewed as the  Hardy-space version of Theorem~\ref{theorem:weak2}.
\begin{theorem}\label{Hardy-spaces-version} Let $0<\alpha<1$. There exists a constant ${\mathrm c}_\alpha$  such that  for every $x \in \mathcal{H}_1^c(\M)$,
\[
\big\| I^\alpha x \big\|_{\mathcal{H}_{1/(1-\alpha)}^c} \leq {\mathrm c}_\alpha\big\|x\|_{\mathcal{H}_1^c}.
\]
\end{theorem}

For the proof, we first establish the following lemma.
It relates the two fractional integrals $I^\alpha$ and $I^{2\alpha}$ when $0<\alpha<1/2$.

\begin{lemma}\label{lemma-singular} Assume that $0<\alpha<1/2$. For every $a \in L_1(\M)$ and $n\geq 1$, we have
\begin{equation}\label{singular-values}
\mu_t\big(S_{c,n}(I^\alpha a) \big) \leq \mu_{t/2}\big(S_{c,n}(I^{2\alpha} a) \big)^{1/2} 
\mu_{t/2}\big(S_{c,n}(a) \big)^{1/2}, \quad t>0.
\end{equation}
\end{lemma}
\begin{proof} Denote by $(e_{ij})$ the canonical matrix of $\mathbb{M}_n$. Then
\begin{align*}
S_{c,n}^2(I^\alpha a) \otimes e_{11}&= \sum_{k=1}^n \zeta_k^{2\alpha } |da_k|^2 \otimes e_{11}\\
&=\big(\sum_{k=1}^n \zeta_k^{2\alpha } |da_k| \otimes e_{1n}\big). \big(\sum_{k=1}^n  |da_k| \otimes e_{n1}\big)\\
&=A.B.
\end{align*} 
Taking singular values relative to  $\M_n \otimes \mathbb{M}_n$, we have 
\[
\mu_t\big(S_{c,n}^2(I^\alpha a) \otimes e_{11}\big) \leq \mu_{t/2}(A). \mu_{t/2}(B)=
\mu_{t/2}(AA^*)^{1/2} \mu_{t/2}(B^*B)^{1/2}.
\]
 Since $AA^*=\sum_{k=1}^n \zeta_k^{4\alpha } |da_k|^2 \otimes e_{11}= S_{c,n}^2(I^{2\alpha}a) \otimes e_{11}$ and 
$B^*B=S_{c,n}^2(a) \otimes e_{11}$, the above inequality translates into 
\[
\mu_t\big(S_{c,n}^2(I^\alpha a) \otimes e_{11}\big)\leq \mu_{t/2}\big(S_{c,n}^2(I^{2\alpha}a) \otimes e_{11}\big)^{1/2} \mu_{t/2}\big(S_{c,n}^2(a) \otimes e_{11}\big)^{1/2}
\]
which is equivalent to \eqref{singular-values}. 
\end{proof}
As a consequence of  Lemma~\ref{lemma-singular},  we may deduce  the next statement.
\begin{lemma}\label{2alpha} Assume that $0<\alpha<1/2$.
For every $a \in \mathcal{H}_1^c(\M)$ and $n\geq 1$,
\[
\big\| (I^\alpha a)_n\big\|_{\mathcal{H}_{1/(1-\alpha)}^c} \leq 2^{(1-\alpha)} \big\| (I^{2\alpha} a)_n\big\|_{\mathcal{H}_{1/(1-2\alpha)}^c}^{1/2} \big\| a\|_{\mathcal{H}_{1}^c}^{1/2}.
\]
\end{lemma}
\begin{proof}
Let $u=1/(1-\alpha)$ and $s=2/(1-2\alpha)$. Then $1/u=1/s +1/2$. Using H\"older's inequality on \eqref{singular-values}, we have 
\[
\big\| \mu_t\big(S_{c,n}(I^\alpha a) \big)\big\|_u \leq \big\|\mu_{t/2}\big(S_{c,n}(I^{2\alpha} a) \big)^{1/2} \big\|_s \ . \  \big\|\mu_{t/2}\big(S_{c,n}(a) \big)^{1/2}  \big\|_2.
\]
This can be easily verified to be equivalent to the statement of the lemma.
\end{proof}

We are now ready to provide the proof of Theorem~\ref{Hardy-spaces-version}.

\medskip

\noindent{\textit{Proof of Theorem \ref{Hardy-spaces-version}}}.  Let $0<\alpha<1$. Fix
$\nu \in \nat$ so that
\[
\frac{1}{2^{\nu +1}} \leq \alpha  <\frac{1}{2^\nu}.
\]
The proof is done by induction on $\nu$.

$\bullet$  $\nu=0$, i.e,  $1/2 \leq \alpha<1$. Let $u=1/(1-\alpha)$. Then $2\leq u$. 
By the noncommutative Khintchine inequalities (\cite{LP4,LPI}), we have for every $n\geq 1$,
\[
\big\|(I^\alpha x)_n\big\|_{\mathcal{H}_u^c}^2 \leq {\mathbb E} \big\| \sum_{k=1}^n \zeta_k^{\alpha } \varepsilon_k dx_k \big\|_u^2
\]
where $(\varepsilon_k)_k$ is a Rademacher sequence and ${\mathbb E}$ denotes  the expectation on the $\varepsilon_k$'s. From  the fact that $L_u(\M)$ is of type 2 (\cite{PX3}), there is a constant $\eta_u$ so that 
\[
\big\|(I^\alpha x)_n\big\|_{\mathcal{H}_u^c}^2  \leq  \eta_u^2 \sum_{k=1}^n \zeta_k^{2\alpha }\big\|dx_k\big\|_u^2.
\]
Now we apply Lemma~\ref{lemma:basic}(i)  to get that since  for every $k\geq 1$, $dx_k \in \mathcal{D}_{k,\infty}$,
\[
\big\|(I^\alpha x)_n\big\|_{\mathcal{H}_u^c}^2 
\leq  \eta_u^2 \sum_{k=1}^n \big\|dx_k\big\|_{1}^2.
\]
Using the  fact that $L_1(\M)$  is of cotype 2 and  another use of the noncommutative Khintchine inequality, we deduce that  there is a constant $\kappa$ so  that 
\[
\big\|(I^\alpha x)_n\big\|_{\mathcal{H}_u^c} \leq  \kappa \eta_u \inf\left\{ \big\|\big(\sum_{k=1}^n a_k^* a_k \big)^{1/2}\big\|_1 +\big\|\big(\sum_{k=1}^n b_k b_k^* \big)^{1/2}\big\|_1  \right\} 
\]
where the infimum is taken over all decompositions $dx_k=a_k +b_k$ for $ k\geq 1$.
A fortiori, we obtain that 
\[
 \big\|(I^\alpha x)_n\big\|_{\mathcal{H}_u^c}  \leq  \kappa \eta_u\big\|x\big\|_{\mathcal{H}_1}\leq \kappa \g_u\big\|x\big\|_{\mathcal{H}_1^c}.
 \]  
 Taking the limit on $n$ proves the case $\nu=0$. 

\medskip

$\bullet$ Assume that the assertion is true for $\nu \geq 0$ and fix $\alpha \in [{2^{-(\nu+2)} } ,  {2^{-(\nu +1)}})$. Note that in this case,  we necessarily have $0<\alpha<1/2$ and therefore  Lemma~\ref{2alpha}  applies. We then have 
${2^{-(\nu+1)} }\leq  2\alpha <  {2^{-\nu}}$ and  thus by assumption there exists a constant ${\mathrm c}_{2\alpha}$ so that 
\[
\big\|(I^{2\alpha} x)_n\big\|_{\mathcal{H}_{1/(1-2\alpha)}^c} \leq {\mathrm c}_{2\alpha}\big\|x\big\|_{\mathcal{H}_1^c}, 
\]
for all $x \in \mathcal{H}_1^c(\M)$ and  all $n\geq 1$. Combining the latter inequality  with Lemma~\ref{2alpha}, we deduce that 
\[
\big\|(I^{\alpha} x)_n\big\|_{\mathcal{H}_{1/(1-\alpha)}^c} \leq 2^{(1-\alpha)}\sqrt{{\mathrm c}_{2\alpha}}\big\|x\big\|_{\mathcal{H}_1^c} 
\]
which  proves that the assertion is true for $\nu+1$. This completes the proof.
\qed

\begin{remarks}\label{row-version} \begin{enumerate} \item[(a)]  Working with  adjoints,  we also have the row-version of Theorem~\ref{Hardy-spaces-version}.

\item[(b)]  For the case $1/2\leq \alpha<1$, the argument  above provides  the stronger statement that
for every $x\in \mathcal{H}_1(\M)$,
\begin{equation}\label{mixed-Hardy}
\big\| I^\alpha x \big\|_{\mathcal{H}_{1/(1-\alpha)}} \leq {\mathrm c}_\alpha\big\|x\|_{\mathcal{H}_1}.
\end{equation}
\end{enumerate}
\end{remarks}

In the next result, we obtain that  \eqref{mixed-Hardy} extends to the full range $0<\alpha<1$.
\begin{corollary}\label{general:hardy}
Let $0<\alpha<1$. There exists a constant ${\mathrm c}_\alpha$ such that  for every $x \in \mathcal{H}_1(\M)$,
\[
\big\| I^\alpha x \big\|_{\mathcal{H}_{1/(1-\alpha)}} \leq {\mathrm c}_\alpha\big\|x\|_{\mathcal{H}_1}.
\]
\end{corollary}
\begin{proof} As noted in Remarks~\ref{row-version}(b),  we only need to consider the case  where $0<\alpha<1/2$. Then  $1<1/(1-\alpha) <2$. Let $x \in \mathcal{H}_1(\M)$ and $\varepsilon>0$. Fix $a\in \mathcal{H}_1^c(\M)$ and $b \in \mathcal{H}_1^r(\M)$ so that:
\begin{enumerate}
\item $x=a+b$;
\item $\big\|a\big\|_{\mathcal{H}_1^c} + \big\|b\big\|_{\mathcal{H}_1^r} \leq \big\|x\big\|_{\mathcal{H}_1} +\varepsilon$.
\end{enumerate}
From Theorem~\ref{Hardy-spaces-version} and Remarks~\ref{row-version}(a),  we have 
\[
\big\| I^\alpha a\big\|_{\mathcal{H}_{1/(1-\alpha)}^c} \leq c_\alpha \big\|a\big\|_{\mathcal{H}_1^c} \ \text{ and }\ \big\| I^\alpha b\big\|_{\mathcal{H}_{1/(1-\alpha)}^r} \leq c_\alpha \big\|b\big\|_{\mathcal{H}_1^r}.
\]  
Taking summations  on both sides of the two previous inequalities give,
\[
\big\| I^\alpha a\big\|_{\mathcal{H}_{1/(1-\alpha)}^c} +\big\| I^\alpha b\big\|_{\mathcal{H}_{1/(1-\alpha)}^r}
\leq c_\alpha\big( \big\|x\big\|_{\mathcal{H}_1} +\varepsilon \big).
\]
The left hand side of the last inequality is clearly larger than $\big\|I^\alpha x\big\|_{\mathcal{H}_{1/(1-\alpha)}}$.  Since $\varepsilon$ is arbitrary, we obtain the desired statement.
\end{proof}
\begin{remark}\label{hardy:BG} By the noncommutative Burkholder-Gundy inequalities (\cite{JX,PX}), Corollary~\ref{general:hardy} is equivalent to the statement that  for $0<\alpha<1$, $I^\alpha$ is bounded from $\mathcal{H}_1(\M)$ into $L_{1/(1-\alpha)}(\M)$. This  is often  easier to apply  when dealing with dualities.
\end{remark}
Before stating the next result, let us recall   the notion  of $\BMO$-spaces for
noncommutative martingales introduced in \cite{PX}. Let
\[
\BMO_C(\M):=\left\{ a \in L^2(\M) : \sup_{n\geq
1}\|\E_n|a-\E_{n-1}a|^2\|_\infty<\infty\right\}. \]
 Then
$\BMO_C(\M)$ becomes a Banach space when equipped with the norm
\[
\|a\|_{\BMO_C}=\Big(\sup_{n\geq
1}\|\E_n|a-\E_{n-1}a|^2\|_\infty\Big)^{1/2}.
\]
Similarly, we define $\BMO_R(\M)$ as the space of all $a$ with $a^*
\in \BMO_C(\M)$ equipped with the natural norm
$\|a\|_{\BMO_R}=\|a^*\|_{\BMO_C}$. The space $\BMO(\M)$ is
the intersection of these two spaces:
\[
\BMO(\M):=\BMO_C(\M) \cap \BMO_R(\M)
\]
with the intersection norm
\[
\|a\|_{\BMO}=\max\big\{\|a\|_{\BMO_C}, \|a\|_{\BMO_R}\big\}.
\]
We recall that as in the classical case, for $1\leq p<\infty$, 
$$\M
\subset \BMO(\M) \subset L_p(\M).$$
 For more information on noncommutative
martingale $BMO$-spaces, we refer to \cite{PX,JX,Mus,JM}.
It was shown in \cite{PX} that the classical  Feffermann     duality is still valid in the noncommutative settings. That is,  we have 
\[
\big(\mathcal{H}_1(\M))^*= \BMO(\M), \,   \text{with equivalent  norms.}
\]

Our next result should be compared with Corollary~\ref{adjoint} above. It is  a direct consequence of Corollary~\ref{general:hardy} and duality.  It is the noncommutative analogue of \cite[Theorem~3(v)]{Chao-Ombe}. See also the next section for more discussions on other items  from \cite[Theorem~3]{Chao-Ombe}.

\begin{corollary}\label{bmo} For $0<\alpha<1$, $I^\alpha$ is bounded from $L_{1/\alpha}(\M)$ into $\BMO(\M)$.
\end{corollary}
We remark that as  shown in Example~\ref{example}, $I^{1/2}$ is not bounded from $L_2[0,1]$ into $L_\infty[0,1]$ for the case of dyadic filtration. Thus,    in general $\BMO(\M)$  in Corollary~\ref{bmo} cannot be replaced by  the von Neumann algebra $\M$.

We also have a  closely  related result  which  follows from Corollary~\ref{bmo} but cannot be directly formulated in the language of fractional integrals since we only defined the latter with $0<\alpha<1$. See the appendix below for more detailed discussions on  this.

\begin{proposition}\label{hardy-bmo} There is an absolute constant  $\kappa$ such that for any   (finite) martingale difference sequence $dx=(dx_k)_{k=1}^n$ in $\mathcal{H}_1(\M)$,
\[
\left\| \sum_{k=1}^n  \zeta_k dx_k \right\|_{\BMO} \leq \kappa  \left\| \sum_{k=1}^n   dx_k \right\|_{\mathcal{H}_1}.
\]
\end{proposition}
\begin{proof} This follows from the  fact that $I^{1/2}$ is bounded simultaneously from 
$\mathcal{H}_1(\M)$ into $L_2(\M)$ and from $L_2(\M)$ into $\BMO(\M)$ and then  use composition.
\end{proof}  

\medskip

We end this section with a  short discussion on  our choice of the scalar sequence  $(\zeta_k)_{k\geq 1}$ introduced in  \eqref{zeta}  and used  in  Definition~\ref{definition:fractional}.  Fix an arbitrary  sequence  of nonnegative  scalars $\nu=(\nu_k)_{k\geq 1}$ and consider    fractional integrals using  $\nu$. We denote this by $I_\nu^\alpha$.  
That is, 
$I_\nu^\alpha x =\sum_{k\geq 1} \nu_k^\alpha dx_k$ for finite martingale $x$. 

For fixed $k\geq 1$ and  $a\in \mathcal{D}_{k,\infty}$,   let $dx=(\delta_{j,k} a)_{j\geq 1}$ where $\delta_{j,k}=0$ for $j\neq k$ and $\delta_{k,k}=1$. Then $dx$ is a martingale difference sequence. If $x$ is the corresponding martingale then it is easy to check  that $\|x\|_{\BMO^c}=\|a\|_\infty$ and $\|x\|_2=\|a\|_2$. Similarly, $\| I_\nu^\alpha x\|_{\BMO^c} =\nu_k \|a\|_\infty$. If $I_{\nu}^{1/2}$ satisfies Corollary~\ref{bmo}, that is,  if  $I_\nu^{1/2}$ is bounded from $L_2(\M)$ into $\BMO(\M)$, then there is an absolute constant $\mathrm{c}$ such that 
\[
\nu_k^{1/2} \|a\|_\infty \leq \mathrm{c} \|a\|_2.
\]
Since this is valid for all $a \in \mathcal{D}_{k,\infty}$, we deduce from \eqref{zeta} that $\zeta_k^{-1/2} \leq \mathrm{c} \nu_k^{-1/2}$. This yields   
$\mathrm{c}^{-2} \nu_k \leq  \zeta_k$ for all $k\geq 1$. In particular, it   shows that (modulo some  constants) our initial choice of $(\zeta_k)_{k\geq 1}$  in Definition~\ref{definition:fractional}  is  the best  possible.


\appendix

\section{ The case $0<p<1$ and open problems}
In this section, we explore the boundedness of fractional integrals when the domain spaces are Hardy spaces indexed by  $p \in (0,1)$.  Our primary tool is the use of  atomic decompositions for martingales. We begin by  recalling the concept of noncommutative atoms  which was introduced in \cite{Bekjan-Chen-Perrin-Y}  for general noncommutative martingales. 
\begin{definition}\label{atom}
Let $0<p <2$. An operator
$a \in L_2 (\M)$ is said to be a $(p,2)_c$-\emph{atom}
with respect to $(\M_n)_{n \geq 1}$,  if there exist $n \geq
1$ and a projection $e \in \M_n$ such that:
\begin{enumerate}[{\rm (i)}]
 \item  $\mathcal{E}_n (a) =0$;
 \item  $ae=a$;
 \item  $ \| a \|_2 \leq \tau (e)^{ 1/2-1/p}$.
\end{enumerate}
Replacing $\mathrm{(ii)}$ by $ \mathrm{(ii)'}~~ea=
a$,  we have the notion of $(p,2)_r$-\emph{atoms}.
\end{definition}

Clearly, $(p,2)_c$-atoms and $(p,2)_r$-atoms are  natural 
noncommutative analogues of  the concept of $(p,2)$-atoms from classical martingales.

Let us now recall the  atomic Hardy spaces for $0<p <2$.

\begin{definition}\label{hardy-atom}
We define the \emph{atomic column  martingale Hardy space} $\h_p^{c, {\at}} (\M)$ as
the  space of all $x \in L_p (\M)$ which admit a decomposition
\[x = \sum_k \lambda_k a_k,
\]
where for each $k$, $a_k$ is a $(p,2)_c$-atom or an element of  the unit ball of   $L_p(\M_1)$,
and $(\lambda_k) \subset  \mathbb{C}$ satisfying $\sum_k |\lambda_k|^p < \infty$.
\end{definition}
We equip $\h_p^{c, {\at}} (\M)$  with the (quasi) norm
\[
\| x \|_{\h_p^{c, \at}} = \inf \Big(\sum_k | \lambda_k |^p\Big)^{1/p},
\]
where the infimum is taken over all decompositions of $x$ described above.

Similarly, we define  the row version $\h_p^{r,\at}(\M )$ as the space of all $x\in L_p(\M)$ for which $x^* \in  \h_p^{c,\at}(\M)$. The space $\h_p^{r,\at}(\M)$ is  equipped  with  the (quasi) norm $\| x\|_{\h_p^{r, \at}}=\|x^*\|_{\h_p^{c,\at}}$.

The atomic Hardy space of noncommutative   martingales is defined as
follows:

\begin{equation*}
\h_p^{\at}(\M)
= \h_p^d(\M) + \h_p^{c,\at} (\M) + \h_p^{r,\at}(\M)
\end{equation*}
equipped with the (quasi) norm
\begin{equation*}
\| x \|_{\h_p^{\at}} =
\inf \big \{\|w\|_{\h_p^d} + \| y\|_{\h_p^{c,at}} + \| z \|_{\h_p^{r,\at}} \big\},
\end{equation*}
where the infimum is taken over all  $w \in \h_p^d(\M)$,
$y \in\h_p^{c,\at} (\M)$, and $z \in \h_p^{r,\at}(\M)$
such that $x = w + y + z$. We refer to \cite{Bekjan-Chen-Perrin-Y, Hong-Mei} for more details  on the concept of atomic decompositions for noncommutative martingales.

\medskip

 One  can describe the dual space of $\h_p^{c,\at}(\M)$ as a \emph{Lipschitz space}. 
For $\beta \geq 0$, we set 
\[
\Lambda_\beta^c(\M)=\big\{x\in L_2(\M):\|x\|_{\Lambda_\beta^c}<\infty\big\}
\]
where
\[
 \|x\|_{ \Lambda_\beta^c}=\max\left\{\|\E_1(x)\|_\infty,\; \sup_{n\ge1}\sup_{e\in\M_n, {\rm projection}}\frac{\big\|(x-\E_n(x))e\big\|_2}{\T(e)^{\beta+1/2}}\right\}\,.
 \]
The space $\Lambda_\beta^c(\M)$ is called the Lipschitz space of order $\beta$. 
For $0<p \leq 1$ and $\beta=1/p-1$, it was shown in  \cite{Bekjan-Chen-Perrin-Y} that
\begin{equation}\label{duality-Lips}
(\h_p^{c,\at}(\M))^*=\Lambda_\beta^c(\M),\   \text{with equivalent  norms.}
\end{equation}
  
Our first result shows that fractional integrals essentially transform  $(p,2)_c$-atoms into $(q,2)_c$-atoms for appropriate values of $p$ and $q$. This could be of independent interest.

\begin{proposition}\label{atoms}
Assume that $0<p<1$, $p<q < 2$, and  $\alpha=1/p-1/q \in (0,1)$. There exists a constant $C_\alpha$ such that if $a$ is  a  $(p,2)_c$-atom then $C_\alpha^{-1} I^\alpha a$ is $(q,2)_c$-atom. In particular,
\[
\big\| I^\alpha a\big\|_{\h_q^{c,\at}} \leq C_\alpha.
\]
\end{proposition}
\begin{proof} Let $a$ be a $(p,2)_c$-atom. There exist $n\geq 1$ and a projection $e \in \M_n$ such that
\begin{itemize}
\item[(i)] $\E_n(a)=0$;
\item[(ii)] $ae=a$;
\item[(iii)] $\|a\|_{2} \leq \T(e)^{1/2-1/p}$.
\end{itemize}
Clearly, $\E_n(I^\alpha a)=0$ and $(I^\alpha a)e=I^\alpha a$. We treat the cases $0<\alpha<1/2$ and $1/2\leq \alpha<1$ separately. 

\smallskip

\noindent Case~1: $1/2 \leq \alpha <1$.  First we note that since $\alpha^{-1} \leq 2$, we have from H\"older's inequality that
 \[\|a\|_{\alpha^{-1}} =\|ae\|_{\alpha^{-1}}\leq \|a\|_2 \T(e)^{\alpha-1/2}.
 \]
   By Corollary~\ref{bmo}, there exists a constant $C_\alpha$ such that 
\begin{equation}\label{alpha3}
C_\alpha^{-1} 
\big\| I^\alpha a\big\|_{\BMO^c} \leq  C_\alpha^{-1} \big\| I^\alpha a\big\|_{\BMO} \leq \big\| a\big\|_{\alpha^{-1}} \leq  \big\|a \big\|_2 \T(e)^{\alpha-1/2} \leq \T(e)^{-1/q}.
\end{equation}
Next, we estimate  the $L_2$-norm of $I^\alpha a$.  Since $\E_n (I^\alpha a)=0$ and $e\in \M_n$,
\[
S_c^2(I^\alpha a)=\sum_{k\geq n} |d_k(I^\alpha a)|^2=e\big(\sum_{k\geq n} |d_k(I^\alpha a)|^2 \big)e.
\]
We have the following estimate:
\begin{align*}
\big\| I^\alpha a\big\|_2^2 &=\T\Big(\sum_{k\geq n} |d_k(I^\alpha a)|^2 e \Big) \\
&=\T\Big(\E_n\big[\sum_{k\geq n} |d_k(I^\alpha a)|^2\big] e \Big) \\
&\leq  \Big\|\E_n\big[\sum_{k\geq n} |d_k(I^\alpha a)|^2\big]\Big\|_\infty \T(e).
\end{align*}
Since $\E_n\big(|I^\alpha a-\E_{n-1}(I^\alpha a)|^2\big)=\E_n\big(\sum_{k\geq n} |d_k(I^\alpha a)|^2\big)$, we deduce that 
\[
\big\| I^\alpha a\big\|_2 \leq \big\| I^\alpha a\big\|_{\BMO^c} \T(e)^{1/2}.
\]
Combining the preceding inequality with \eqref{alpha3}, we conclude that $C_\alpha^{-1}\big\|I^\alpha a \big\|_2 \leq  \T(e)^{1/2-1/q}$ which is equivalent  to $C_\alpha^{-1} I^\alpha a$ being  a $(q,2)_c$-atom.

\medskip

\noindent Case~2: $0<\alpha<1/2$.  Fix $1<r <2$ so that $\alpha=1/r-1/2$. By Corollary~\ref{theorem:main},   $I^\alpha$ is bounded from $L_r(\M)$ into $L_2(\M)$. Taking adjoint, $I^\alpha : L_2(\M) \to L_{r'}(\M)$ is bounded where $1/r +1/{r'}=1$. There exists a constant $C_\alpha$ such that
\begin{equation}\label{r-prime}
C_\alpha^{-1} \big\| I^\alpha a \big\|_{r'} \leq \big\|a\big\|_2 \leq \T(e)^{1/2-1/p}.
\end{equation}
On the other hand, from H\"older's inequality,  we have
\begin{align*}
\big\| I^\alpha a\big\|_2^2 &=\T\big(|I^\alpha a|^2 e \big)\\
&\leq \T\big(|I^\alpha a|^{r'} \big)^{2/{r'}} \T(e)^{1-(2/{r'})}.
\end{align*}
Therefore,
\begin{equation}\label{r-prime2}
\big\| I^\alpha a\big\|_2 \leq \big\| I^\alpha a\big\|_{r'} \T(e)^{1/2 -1/{r'}}.
\end{equation}
Combining \eqref{r-prime} and \eqref{r-prime2}, we get $C_\alpha^{-1}\big\|I^\alpha a\big\|_2 \leq \T(e)^{1-1/p-1/{r'}}=\T(e)^{1/r -1/p}$. From the choice of $r$ above, we have $1/r -1/p=1/2 -1/q$. That is,  $C_\alpha^{-1}\big\|I^\alpha a\big\|_2 \leq \T(e)^{1/2-1/q}$, which again shows that $C_\alpha^{-1} I^\alpha a$ is a $(q,2)_c$-atom.
\end{proof}

The next  theorem is an   immediate consequence of Proposition~\ref{atom}  and duality. We leave the details  of its proof to the reader.

\begin{theorem}\label{quasi-case}
\begin{itemize}
\item[(a)] Assume that  $0<p<1$, $p<q<2$, and $\alpha=1/p -1/q \in (0,1)$. Then 
$I^\alpha$ is bounded from  $\h_p^{c,\at}(\M)$ into $\h_q^{c,\at}(\M)$.
\item[(b)]  For $0<\beta$ and $0<\alpha <1$, $I^\alpha$ is bounded from  $\Lambda_\beta^c(\M)$ into $ \Lambda_{\beta +\alpha}^c(\M)$.  
\end{itemize}
\end{theorem}

So far  we considered   only fractional integrals of order $\alpha$  under the assumption that $0<\alpha <1$. Indeed, all the results  stated in Section~2  do require this assumption. In fact,  both the statements   and techniques of proofs used in Theorem~\ref{theorem:weak2}, Theorem~\ref{main:maximal}, and Theorem~\ref{Hardy-spaces-version}  highlighted  the need for  $1-\alpha$ to be nonegative.   However, when considering the case $0<p<1$, the situation is different. For instance, both statements  in Theorem~\ref{quasi-case} still  make sense without the assumption $0<\alpha<1$.
To avoid any potential confusion, we will introduce different  notation for the general case. Let $\gamma >0$. We  denote by $\tilde{I}^\gamma$  the transformation defined by setting for any  martingale $x=(x_n)_{n\geq 1}$,
$\tilde{I}^\gamma x =  \{(\tilde{I}^\gamma x)_n\}_{n\geq 1}$  where for $n\geq 1$,
\[
(\tilde{I}^\gamma x)_n =\sum_{k=1}^n \zeta_k^{\gamma} dx_k.
\]
Clearly, $\tilde{I}^\gamma$ is simply $I^\gamma$ when $0<\gamma<1$. Moreover, when $\gamma\geq 1$,  set  $n(\gamma) :=\lfloor{\gamma}\rfloor +1$  where $\lfloor \cdot \rfloor$ denotes  the greatest integer function and $\alpha(\gamma):=\gamma/n(\gamma)$. Then clearly $0<\alpha(\gamma) <1$ and we may view $\tilde{I}^\gamma$ as the compositions of the fractional integral $I^{\alpha(\gamma)}$ with itself $n(\gamma)$-times. We  now consider boundedness  properties of   $\tilde{I}^\gamma$ as  a linear transformation. The following theorem should be compared with \cite[Theorem~3]{Chao-Ombe}.

\begin{theorem}\label{hardy-Lip-bmo}
\begin{enumerate}[{\rm (i)}]
\item  If $0<\beta$ and $0<\gamma$,  then $\tilde{I}^\gamma$ is bounded from  $\Lambda_\beta^c(\M)$ into $ \Lambda_{\beta + \gamma}^c(\M)$.
\item If $0<p<q<\infty$ and $\gamma=1/p-1/q$,  then $\tilde{I}^\gamma$  is bounded from  $\h_p^{c,\at}(\M)$ into $\mathcal{H}_q^c(\M)$.
\item For   every $\gamma>0$, $\tilde{I}^\gamma$ is bounded from $\BMO^c(\M)$ into  $\Lambda_\gamma^c(\M)$.
\item If $1<p<\infty$ and $\gamma>1/p$, then $\tilde{I}^\gamma$ is bounded from
$L_p(\M)$ into $\Lambda_{\gamma-(1/p)}^c(\M)$.
\item If $\gamma>1$,  then $\tilde{I}^\gamma$ is bounded from $\mathcal{H}_1(\M)$ into 
$\Lambda_{\gamma-1}^c(\M)$.
\end{enumerate}
\end{theorem}

\begin{proof} $\bullet$  Item $(i)$ is already the second part of Theorem~\ref{quasi-case} if $0<\gamma<1$.  Assume that $\gamma\geq 1$ and let $\alpha(\gamma)$ and $n(\gamma)$ be as described above. From  the second part of Theorem~\ref{quasi-case},  $I^{\alpha(\gamma)}$ is bounded from  $\Lambda_{\beta +(k-1)\alpha(\gamma)}^c(\M)$ into $ \Lambda_{\beta +k\alpha(\gamma)}^c(\M)$ for all
 integers $k \in [1, n(\gamma)]$. We apply   $I^{\alpha(\gamma)}$ successively $n(\gamma)$-times and get $\tilde{I}^\gamma$ as the composition:
\begin{equation*}
\tilde{I}^\gamma: \Lambda_{\beta}^c(\M) \xrightarrow{I^{\alpha(\gamma)}} \Lambda_{\beta +\alpha(\gamma)}^c(\M) \xrightarrow{I^{\alpha(\gamma)}}   \Lambda_{\beta +2\alpha(\gamma)}^c(\M)\dots \Lambda_{\beta + \gamma-\alpha(\gamma) }^c(\M) \xrightarrow{I^{\alpha(\gamma)}}   \Lambda_{\beta +\gamma}^c(\M).
\end{equation*}   
This shows that $ \tilde{I}^\gamma: \Lambda_{\beta}^c(\M) \to \Lambda_{\beta +\gamma}^c(\M)$ is bounded.

$\bullet$ First, we note that $(ii)$ follows directly from $(i)$ by duality when $0<p<q<1$. So we assume that  $q\geq 1$.  Fix $\varepsilon>0$ such that
$0<p<1-\varepsilon$. Let $\gamma_1 :=1/p-1/(1-\varepsilon)$, $\gamma_2:=1/(1-\varepsilon) -1$, and $\gamma_3:=1-1/q$. We have $\tilde{I}^{\gamma_1}: \h_p^{c,\at}(\M) \to \h_{1-\varepsilon}^{c,\at}(\M)$ is bounded.  Also since $0<\gamma_2<1$, we have from the first part of Theorem~\ref{quasi-case} that $I^{\gamma_2}: \h_{1-\varepsilon}^{c,\at}(\M) \to \h_1^{c,\at}(\M)$ is bounded. From the inclusion $\h_1^{c,\at}(\M) \subset \mathcal{H}_1^c(\M)$ for which we refer to \cite[Proposition~2.2]{Bekjan-Chen-Perrin-Y}, we may also state  that $I^{\gamma_2}: \h_{1-\varepsilon}^{c,\at}(\M) \to \mathcal{H}_1^c(\M)$ is bounded. Furthermore, since $0<\gamma_3<1$, it follows from Theorem~\ref{Hardy-spaces-version} that $I^{\gamma_3}: \mathcal{H}_1^c(\M) \to \mathcal{H}_q^c(\M)$ is bounded. We can then conclude that  the composition $\tilde{I}^\gamma=I^{\gamma_3} I^{\gamma_2} \tilde{I}^{\gamma_1}$ is bounded from $\h_p^{c,\at}(\M)$ into $\mathcal{H}_q^c(\M)$. 
For the case where  $q=1$, we only need to consider $\gamma_1$ and $\gamma_2$.

$\bullet$  Item $(iii)$ is an immediate consequence of $(ii)$ by duality and $q=1$.

$\bullet$ For $(iv)$,   we observe from Corollary~\ref{bmo} that if $\alpha=1/p$ then 
$I^\alpha: L_p(\M) \to \BMO(\M)$ is bounded. A fortiori, $I^\alpha: L_p(\M)  \to \BMO^c(\M)$ is bounded. On the other hand,  we also have  from $(iii)$ that $\tilde{I}^{\gamma-\alpha}: \BMO^c(\M) \to \Lambda_{\gamma-\alpha}^c(\M)$ is bounded. Thus, $(iv)$ follows by taking composition.

$\bullet$  Item $(v)$  follows from  combining Proposition~\ref{hardy-bmo} and $(iii)$. Indeed, from Proposition~\ref{hardy-bmo}, $\tilde{I}^1: \mathcal{H}_1(\M) \to \BMO^c(\M)$  is bounded and  from $(iii)$, $\tilde{I}^{\gamma-1}: \BMO^c(\M) \to  \Lambda_{\gamma-1}^c(\M)$ is bounded.
\end{proof}

\begin{remark}
All results stated in Theorem~\ref{atoms} and Theorem~\ref{hardy-Lip-bmo} are valid for the corresponding row-versions.
\end{remark}
For the case of mixed  Hardy spaces, we  may also state:
\begin{corollary}
If $0<p<q<\infty$ and $\gamma=1/p-1/q$, then $\tilde{I}^\gamma$ is bounded from $\h_p^{\at}(\M)$  into $\mathcal{H}_q(\M)$.
\end{corollary}
\begin{proof} It is enough to prove that $\tilde{I}^\gamma$ is bounded from $\h_p^d(\M)$ into $\h_q^d(\M)$. In view  of the proof  of Theorem~\ref{hardy-Lip-bmo}(ii), it suffices to verify the special case where $0<p<q\leq 1$ and $\gamma \in (0,1)$. This can be deduced from the following claim: for every $k\geq 1$,
\[
\zeta_k^{\gamma q} \big\| a \big\|_q^q \leq  \big\|a\big\|_p^p, \quad \text{for all}\ a\in \M_k.
\]
To prove this claim, we apply Lemma~\ref{lemma:basic}~(i) to $|a|^p$ and any given $0<\beta<1$    to get
\[
\zeta_k^{\beta/(1-\beta)} \T\big( |a|^{p/(1-\beta)} \big) \leq \T\big(|a|^p \big).
\]
Choose $\beta$ such that $p/(1-\beta)=q$. One can easily verify that $\beta/(1-\beta)=\gamma q$. This proves the claim.  
\end{proof}

\medskip

In \cite[Theorem3]{Chao-Ombe}, the classical dyadic filtration  was handled without specifically referring to atomic decompositions or atomic Hardy  spaces. We do not know if the use of atoms and more specifically the use of Proposition~\ref{atoms} can be avoided. More precisely, we do not know if the atomic Hardy spaces in  the statements of Theorem~\ref{quasi-case} and Theorem~\ref{hardy-Lip-bmo}  can be replaced with the usual Hardy spaces. This of course is  closely connected to the problem of atomic decomposition for noncommutative martingales.
We leave this as an  open question. 
\begin{problem}
Is $\tilde{I}^\gamma$ bounded from $\mathcal{H}_p^c(\M)$ into $\mathcal{H}_q^c(\M)$ when $0<p<q \leq 1$ and $\gamma=1/p -1/q$?
\end{problem}

To complete this circle of ideas, we consider  maximal Hardy spaces.  Let us recall the classical Davis theorem (see \cite{Da2}) that states that for every commutative martingale  $x \in \mathcal{H}_1$ then 
$\|x\|_{\mathcal{H}_1} \sim \| \sup_k|\E_k(x)| \|_1$. Following the ideas described in 
Subsection~\ref{weak-type},  one can  define the maximal Hardy space of noncommutative martingales $\mathcal{H}_1^{\rm max}(\M)$ as the space of all martingale $x \in L_1(\M)$ for which $\|x\|_{\mathcal{H}_1^{\rm max}}=\| (\E_k(x))_{k}\|_{L_1(\M; \ell_\infty)}$ is finite.  The Davis theorem stated above  is equivalent to say that for the commutative case, the two Hardy spaces  $\mathcal{H}^1$ and $\mathcal{H}_1^{\rm max}$ coincide. Unfortunately, Davis theorem does not extend to the noncommutative case.  Indeed, it was shown in \cite[Corollary~14]{JX2}
that $\mathcal{H}_1$ and $\mathcal{H}_1^{\rm max}$ do not coincide in general. The following  problem was motivated by Theorem~\ref{Hardy-spaces-version}:
\begin{problem} Let $0<\alpha<1$. Is $I^\alpha$  bounded  from $\mathcal{H}_1^{\rm max}(\M)$  into $\mathcal{H}_{1/(1-\alpha)}(\M)$?
\end{problem}


\bigskip

\noindent{\bf Acknowledgements}
This work was completed when the second-named author was visiting Miami University. She would like to express her gratitude to the Department of Mathematics of Miami University for its warm hospitality.

\providecommand{\bysame}{\leavevmode\hbox to3em{\hrulefill}\thinspace}
\providecommand{\MR}{\relax\ifhmode\unskip\space\fi MR }
\providecommand{\MRhref}[2]{%
  \href{http://www.ams.org/mathscinet-getitem?mr=#1}{#2}
}
\providecommand{\href}[2]{#2}

\end{document}